\documentclass[12pt, a4paper]{article}

\usepackage[utf8]{inputenc}
\usepackage[T1]{fontenc}
\usepackage{amsmath, amsthm, amssymb, amsfonts}
\usepackage{geometry}
\usepackage{hyperref}
\usepackage{listings}
\usepackage{xcolor}
\usepackage{graphicx}
\usepackage{bbm}
\newtheorem{lemma}{Lemma}
\newtheorem{theorem}{Theorem}

\newcommand{\stirlingii}[2]{\genfrac{\{}{\}}{0pt}{}{#1}{#2}}

\title{\textbf{Extended Central Factorial Numbers and the Flickering Operator}}
\author{Andrii Husiev}
\date{\today}

\begin{document}

\maketitle

\begin{abstract}
This paper introduces a class of extended central factorial numbers generated by a parity-dependent recurrence relation, termed the "flickering operator". We demonstrate that the resulting triangular structure, now indexed as OEIS A395021, provides a unified recursive framework for alternating bit sequences (A000975) and normalized tangent-secant coefficients (A036969). This study provides an alternative integer-based expansion for power sums. While similar to the central factorial methods explored by Knuth \cite{knuth1993}, our flickering basis offers an integrated computational scheme that avoids fractional Bernoulli numbers by construction. We provide explicit closed-form expressions, discuss its geometric derivation from finite difference tables, and present a full Python implementation.

Structural Synthesis. A key contribution of this work is the unification of previously disparate combinatorial sequences into a single coherent framework. While certain columns of the flickering triangle $T(n,k)$ (such as A008957) could be partially retrieved from the diagonals of existing central factorial arrays, our structure provides a complete representation including previously unindexed even-positioned terms. Furthermore, the row-wise analysis reveals that the flickering operator generates full integer sequences where previously only the odd-indexed elements (e.g., A002451) were identified. This synthesis bridges the gap between these sequences, positioning A395021 as the underlying master structure.

\end{abstract}

\section{Introduction}

The calculus of finite differences and the theory of Stirling numbers of the second kind $\stirlingii{n}{k}$ have long served as the cornerstone of discrete mathematics and combinatorial enumeration \cite{stirling, comtet}. Classical partition theory and power sequence expansions heavily rely on these numbers to transition between ordinary powers and falling factorials.

A natural extension of this framework is the study of central factorial numbers, which arise in the expansion of polynomials in shifted bases and are intimately connected to the derivatives of trigonometric functions. The operational methods and combinatorial identities involving shifted finite differences were systematically organized by Riordan \cite{riordan} and Carlitz \cite{carlitz}, who formalized the structural properties and arithmetic congruences of central factorial numbers. Furthermore, Knuth \cite{knuth1993} demonstrated the profound utility of central factorial numbers in providing elegant closed-form solutions to Faulhaber’s formula for sums of powers.

Despite these advancements, the classical transition to summing powers typically necessitates the introduction of fractional coefficients, specifically the Bernoulli numbers. Moreover, as cataloged extensively in the On-Line Encyclopedia of Integer Sequences (OEIS) \cite{sloane}, many structural sequences related to central factorials (such as A008957 and A036969) often present isolated or fragmented combinatorial profiles.

To bridge the gap between discrete partitions, trigonometric series expansions, and strictly integer-valued power sums, we modify the classical recurrence. By accounting for the parity of the step, we define the "flickering operator". This operator acts as a symmetry-breaking filter on the power sequence $m^n$, yielding the unified triangular structure \textbf{A395021}. 

The name \textit{'Extended central factorial numbers of the second kind'} is proposed for this triangular structure, as it generalizes the classical central factorial numbers via a parity-dependent recurrence. As we will demonstrate, this approach not only unifies several previously disparate sequences but also offers a completely integer-based expansion basis for Faulhaber's formula, avoiding fractional terms by construction.

\textbf{Definition.} We define the triangular structure $T(n, k)$ as the \textit{normalized flickering central finite differences}. For a fixed $n \ge 1$ and $1 \le k \le n$, the terms are given by:
\begin{equation}
T(n, k) = \frac{1}{k!} \delta^k f_i(0) \quad \text{where} \quad f(j) = j^n
\end{equation}
The selection index $i$ follows the \textbf{parity-dependent rule}:
\begin{equation}
i = 
\begin{cases} 
1/2, & \text{if } k \text{ is odd} \\
0, & \text{if } k \text{ is even}
\end{cases}
\end{equation}
This alternating selection ensures that all resulting coefficients are integers, effectively interleaving the structural properties of central factorial numbers and their scaling transitions.

\section{The Flickering Recurrence and Array A394582}

\subsection{Numerical Representation}
The initial values of the array demonstrate the interleaved scaling behavior (where $n$ is the row index and $k$ is the column index):
\begin{table}[ht]
\centering
\renewcommand{\arraystretch}{1.2}
\begin{tabular}{c|cccccccc}
$n \setminus k$ & 1 & 2 & 3 & 4 & 5 & 6 & 7 & 8 \\ \hline
1 & 1 & 1 & 1 & 1 & 1 & 1 & 1 & 1 \\
2 & 1 & 2 & 5 & 10 & 21 & 42 & 85 & 170 \\
3 & 1 & 3 & 14 & 42 & 147 & 441 & 1408 & 4224 \\
4 & 1 & 4 & 30 & 120 & 627 & 2508 & 11440 & 45760 \\
5 & 1 & 5 & 55 & 275 & 2002 & 10010 & 61490 & 307450 \\
\end{tabular}
\caption{The array $Todd(n,k)$ as indexed in OEIS A394582.}
\end{table}

\section{Explicit Formulas and Identities}

\subsection{Central Finite Difference Representation for A394582}

In this section, we analyze the specific structure of the array \textbf{OEIS A394582}. This array is constructed as a sub-grid of the master flickering central factorial triangle (A395021), formed by extracting only its odd-indexed columns. Consequently, the row index $n$ in \textbf{A394582} corresponds to the order of differences $2n-1$ in the master difference table.

The terms of this sub-array can be expressed as the $(2n-1)$-th finite difference of the power sequence $f(j) = j^m$, where the exponent is $m = 2n + k - 2$. The difference is evaluated at the central offset $j = -(n-1)$ and normalized by $(2n-1)!$:
\begin{equation}
Todd(n, k) = \frac{1}{(2n-1)!} \sum_{i=0}^{2n-1} (-1)^{2n-1-i} \binom{2n-1}{i} (i-n+1)^{2n+k-2}
\end{equation}
This identity positions the array \textbf{A394582} as a vertical slice of the master structure, explaining its specific growth properties compared to the unified triangle.

\subsection{Recurrence Relation}
The array $Todd(n,k)$ is defined for $n \ge 1, k \ge 1$ with boundary conditions $Todd(n,1) = 1$ and $Todd(1,k) = 1$. The flickering operator defines the transition for $n, k > 1$:
\begin{equation}
Todd(n, k) = 
\begin{cases} 
n \cdot Todd(n, k-1) + Todd(n-1, k), & \text{if } k \equiv 1 \pmod{2} \\
n \cdot Todd(n, k-1), & \text{if } k \equiv 0 \pmod{2}
\end{cases}
\label{eq:flickering_rule}
\end{equation}

\subsection{Representation via Stirling Numbers}
Alternatively, the array can be represented using Stirling numbers of the second kind. This form highlights the combinatorial partitions underlying the flickering operator:
\begin{equation}
Todd(n, k) = \sum_{j=0}^{2n+k-2} \binom{2n+k-2}{j} (1-n)^{2n+k-2-j} \stirlingii{j}{2n-1}
\end{equation}

\section{Row-wise Analysis (\texorpdfstring{$n = \text{const}$})}
The rows of array $Todd(n, k)$ represent the evolution of the flickering operator for fixed scaling factors $n$.
\begin{itemize}
    \item \textbf{Row $n=1$:} $(1, 1, 1, \dots)$ correspond to constant sequence \textbf{A000012}.
    \item \textbf{Row $n=2$:} $(1, 2, 5, 10, 21, 42, \dots)$ correspond to Lichtenberg sequence \textbf{A000975}, also odd elements correspond to \textbf{A002450}.
    \item \textbf{Row $n=3$:} $(1, 3, 14, 42, 147, \dots)$ correspond to \textbf{A394882}, identified in this study via the central difference table, also odd elements correspond to \textbf{A002451}.   
    \item \textbf{Row $n=4$:} $(1, 4, 30, 120, 627, \dots)$ currently not indexed in OEIS, but odd elements correspond to \textbf{A383838}.   
    \item \textbf{Row $n=5$:} $(1, 5, 55, 275, 2002, \dots)$ currently not indexed in OEIS, but odd elements correspond to \textbf{A383840}.
\end{itemize}

\subsection{Generating Functions for Row Sequences}

The structural regularity of the flickering operator is further evidenced by the closed-form generating functions for the rows of array \textbf{A394582}. We distinguish between the sequences formed by odd-indexed elements and the complete row sequences.

\subsubsection{Odd-indexed Subsequences}
When considering only the odd-indexed elements $k$ (corresponding to the sequences A002450, A002451, A383838, and A383840), the generating functions $G_n^{odd}(x)$ follow a elegant reciprocal structure of square-indexed factors:

\begin{itemize}
    \item \textbf{Row $n=2$ (A002450):} $G_2^{odd}(x) = \frac{x}{(1 - x)(1 - 4x)}$
    \item \textbf{Row $n=3$ (A002451):} $G_3^{odd}(x) = \frac{x}{(1 - x)(1 - 4x)(1 - 9x)}$
    \item \textbf{Row $n=4$ (A383838):} $G_4^{odd}(x) = \frac{x}{(1 - x)(1 - 4x)(1 - 9x)(1 - 16x)}$
    \item \textbf{Row $n=5$ (A383840):} $G_5^{odd}(x) = \frac{x}{(1 - x)(1 - 4x)(1 - 9x)(1 - 16x)(1 - 25x)}$
\end{itemize}

In general, for the $n$-th row subsequence, the generating function is given by:
\begin{equation}
G_n^{odd}(x) = \frac{x}{\prod_{j=1}^{n} (1 - j^2 x)}
\end{equation}

\subsubsection{Full Row Sequences of A394582}
The complete rows of array \textbf{A394582}, including both even and odd terms, exhibit the "flickering" effect through the introduction of $x^2$ in the denominator and a linear numerator $(1 + nx)$. This reflects the interleaved nature of the scaling transitions:

\begin{itemize}
    \item \textbf{Row $n=2$:} $G_2(x) = \frac{x(1 + 2x)}{(1 - x^2)(1 - 4x^2)}$
    \item \textbf{Row $n=3$:} $G_3(x) = \frac{x(1 + 3x)}{(1 - x^2)(1 - 4x^2)(1 - 9x^2)}$
    \item \textbf{Row $n=4$:} $G_4(x) = \frac{x(1 + 4x)}{(1 - x^2)(1 - 4x^2)(1 - 9x^2)(1 - 16x^2)}$
    \item \textbf{Row $n=5$:} $G_5(x) = \frac{x(1 + 5x)}{(1 - x^2)(1 - 4x^2)(1 - 9x^2)(1 - 16x^2)(1 - 25x^2)}$
\end{itemize}

The general form for the generating function of the $n$-th row of A394582 is:
\begin{equation}
G_n(x) = \frac{x(1 + nx)}{\prod_{j=1}^{n} (1 - j^2 x^2)}
\end{equation}

This representation demonstrates that the flickering operator acts as a symmetry-breaking filter on the classical central factorial space, where the factor $(1 + nx)$ compensates for the parity-dependent scaling of the finite differences.

\section{Column-wise Analysis \texorpdfstring{($k = \text{const}$)}{(k = const)}}

The columns of the array $Todd(n,k)$ demonstrate structural consistency through their ties to classical sequences.
\begin{itemize}
    \item \textbf{Column $k=1$:} $(1, 1, 1, 1, 1, \dots)$ correspond to \textbf{A000012}.
    \item \textbf{Column $k=2$:} $(1, 2, 3, 4, 5, \dots)$ correspond to natural numbers \textbf{A000027}.
    \item \textbf{Column $k=3$:} $(1, 5, 14, 30, 55, \dots)$ correspond to \textbf{Square Pyramidal Numbers} (\textbf{A000330}).
    \item \textbf{Column $k=4$:} $(1, 10, 42, 120, 275, \dots)$ correspond to \textbf{Kekul\'{e} numbers for certain benzenoids} (\textbf{A108678}).
    \item \textbf{Column $k=5$:} $(1, 21, 147, 627, 2002, \dots)$ correspond to \textbf{A060493}.
  \item \textbf{Column $k=6$:} $(1, 42, 441, 2508, 10010, \dots)$ — correspond to \textbf{A394883}, identified in this study via the central difference table.
    \item \textbf{Column $k=7$:} $(1, 85, 1408, 11440, 61490, \dots)$ correspond to \textbf{A351105}.
    \item \textbf{Column $k=8$:} $(1, 170, 4224, 45760, 307450, \dots)$ — correspond to \textbf{A394896}, identified in this study via the central difference table.
    \item \textbf{Column $k=9$:} $(1, 341, 13013, 196053, 1733303, \dots)$ correspond to \textbf{A353021}.
\end{itemize}

\subsection{Polynomial Representation of the Array Columns}

The columns of the sub-array \textbf{A394582} admit a strict closed-form polynomial representation in $n$. For any odd column index $k = 2m+1$ (where $m \ge 0$), the sequence values $Todd(n, k)$ factorize into a product of a standard base polynomial $T_m(n)$ and a non-trivial rational function $U_m(n)$:
\begin{equation}
Todd(n, 2m+1) = T_m(n) \cdot U_m(n)
\end{equation}

The standard base polynomial $T_m(n)$ encompasses the predictable roots of the discrete integration and is explicitly given by the product:
\begin{equation}
T_m(n) = \left( \prod_{i=0}^{m} (n+i) \right) \left( \prod_{j=1}^{m} (2n+2j-1) \right)
\end{equation}

The non-trivial component $U_m(n) = \frac{P_m(n)}{D_m}$ consists of a polynomial $P_m(n)$ and a constant denominator $D_m$. The denominator $D_m$ follows a growth pattern governed by powers of 6 and factorials, serving as the normalization constant for the repeated summation of squares.

Due to the fundamental column transition rule $Todd(n, 2m+1) - Todd(n-1, 2m+1) = n^2 \cdot Todd(n, 2m-1)$, the polynomial $U_m(n)$ is uniquely determined by the differential-difference equation:
\begin{equation}
T_m(n) U_m(n) - T_m(n-1) U_m(n-1) = n^2 T_{m-1}(n) U_{m-1}(n)
\end{equation}

This recurrence dictates that the $U$-polynomial for any arbitrary column can be directly derived from the polynomial of the preceding odd column, entirely bypassing the need to compute the full two-dimensional array.

\subsection{Geometric Analysis of the Finite Difference Table}
The derivation of the array $T(n,k)$ is visually encoded in the finite difference table of the power sequence $m^k$ (see Figure~\ref{fig:finite_table}).

\begin{figure}[htbp]
    \centering
    \includegraphics[width=0.8\textwidth]{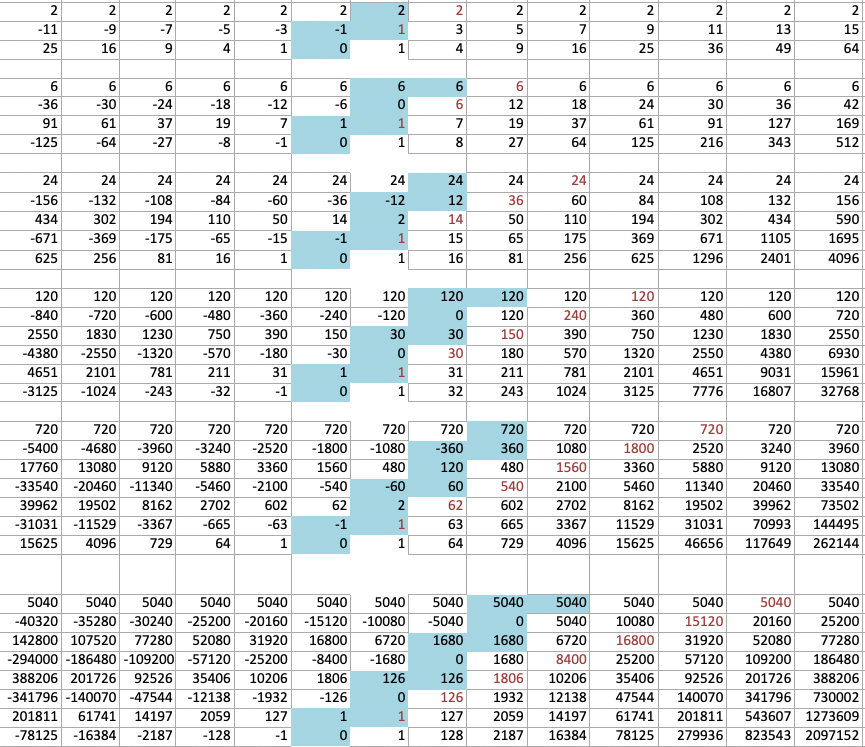} 
    \caption{Finite difference table with the blue selection path for $T(n,k)$ and red highlights for Stirling numbers.}
    \label{fig:finite_table}
\end{figure}

\begin{itemize}
     \item \textbf{The Selection Path (Blue Highlights):} This trajectory represents the unnormalized precursors of array \textbf{A394582}. Once divided by $(2n-1)!$, they yield the integer coefficients. For example, large differences (like $5040$) are transformed into structured terms of $T(n,k)$ through this central scaling.
    \begin{sloppypar}
    Furthermore, the sum of these normalized central differences along each anti-diagonal yields the sequence $1, 2, 2, 5, 7, 21, 37, 126, 264, 1001, \dots$, where the odd-indexed terms correspond to the generalized Bell numbers (OEIS \textbf{A135920}), which are the row sums of A036969.
    \end{sloppypar}
    \item \textbf{Classical Stirling Numbers (Red Highlights):} The values highlighted in red correspond to unnormalized Stirling numbers $n! \stirlingii{k}{n}$.
\end{itemize}

The symmetry of the table is maintained around a central axis. The alignment of the blue selection path explains the transition from alternating differences to the structured coefficients.

\subsection{The Selection Path and Triangle Construction}
As illustrated in Figure~\ref{fig:finite_table}, the normalized flickering central finite differences triangle $T(n,k)$ is constructed by systematically extracting values from the central symmetry axis of the finite difference table for each power $n \in \{1, 2, 3, \dots\}$. 

For a fixed $n$, we collect the $k$-th order differences $\Delta^k_i$ for $1 \le k \le n$. To ensure integer normalization by $k!$, the selection index $i$ alternates:
\begin{itemize}
    \item For \textbf{odd $k$}: we select the central difference at the half-integer offset $i = -1/2$.
    \item For \textbf{even $k$}: we select the central difference at the integer offset $i = 0$.
\end{itemize}

The resulting values, divided by $k!$, form the $n$-th row of the triangle. The first 10 rows generated by this extraction process are presented below:

\begin{center}
\small
\begin{tabular}{l}
$n=1: 1;$ \\
$n=2: 1, 1;$ \\
$n=3: 1, 0, 1;$ \\
$n=4: 1, 1, 2, 1;$ \\
$n=5: 1, 0, 5, 0, 1;$ \\
$n=6: 1, 1, 10, 5, 3, 1;$ \\
$n=7: 1, 0, 21, 0, 14, 0, 1;$ \\
$n=8: 1, 1, 42, 21, 42, 14, 4, 1;$ \\
$n=9: 1, 0, 85, 0, 147, 0, 30, 0, 1;$ \\
$n=10: 1, 1, 170, 85, 441, 147, 120, 30, 5, 1;$
\end{tabular}
\end{center}

\section{The normalized flickering central finite differences Triangle \texorpdfstring{$T(n,k)$}{T(n,k)} (A395021)}
The complete structural evolution of the flickering operator is encapsulated in the normalized flickering central finite differences triangle $T(n,k)$, where $1 \le k \le n$. This triangle unifies the previous observations by providing the full set of normalized central differences $f^k_i/k!$ with alternating indices $i \in \{1/2, 0\}$.

\subsection{The Merging Effect: Beyond A008957}
The unified triangle $T(n,k)$ effectively acts as a bridge between column-wise and row-wise growth patterns. In previous studies, these patterns appeared as isolated properties:
\begin{itemize}
    \item \textbf{Columns:} Many columns were only accessible as specific diagonals of the central factorial triangle (A008957). Our model provides a direct recursive path to these sequences as fixed columns.
    \item \textbf{Rows:} Prior to this study, many row-wise sequences were recognized only through their odd-indexed terms (e.g., $n=2$ row and A000975). The flickering operator "fills in the gaps," providing a rigorous definition for the intermediate (even-indexed) terms, thus completing the combinatorial profile of the array.
\end{itemize}

\subsection{New Closed-Form Identities for A008957}
As a direct consequence of the structural equivalence $A008957(n, k) = T(2n-1, 2n-2k+1)$, we derive two new explicit representations for the central factorial numbers of the second kind. These identities extend the analytical toolkit for A008957:

\begin{enumerate}
    \item \textbf{Finite Difference Representation:}
    \begin{equation}
    A008957(n, k) = \frac{1}{(2n-2k+1)!} \sum_{i=0}^{2n-2k+1} (-1)^{2n-2k+1-i} \binom{2n-2k+1}{i} (i-n+k)^{2n-1}
    \end{equation}
    
    \item \textbf{Representation via Stirling Numbers of the Second Kind:}
    \begin{equation}
    A008957(n, k) = \sum_{j=0}^{2n-1} \binom{2n-1}{j} (k-n)^{2n-1-j} \stirlingii{j}{2n-2k+1}
    \end{equation}
\end{enumerate}
These formulas provide a way to calculate central factorial numbers directly from power sequence differences and binomial expansions without requiring the full recurrence.

\subsection{Parity-Dependent Recurrence Relation}
While the direct calculation involves alternating sums of powers, the triangle is uniquely characterized by a four-stage autonomous recurrence. For $1 < k < n$, the transitions are governed by the parity of both the step $n$ and the order $k$:

\begin{itemize}
    \item \textbf{Even $k$, Odd $n$:} $T(n, k) = 0$.
    \item \textbf{Even $k$, Even $n$:} $T(n, k) = \frac{2}{k} T(n, k-1)$.
    \item \textbf{Odd $k$, Even $n$:} $T(n, k) = \frac{k+1}{2} T(n-1, k)$.
    \item \textbf{Odd $k$, Odd $n$:} $T(n, k) = \frac{k+1}{2} T(n-1, k) + \frac{2}{k-1} T(n-1, k-2)$.
\end{itemize}
The boundary conditions are $T(n, 1) = 1$ and $T(n, n) = 1$.

\subsection{Proof of Integrality}

\begin{lemma}[Integrality of the Flickering Triangle]
The terms $T(n, k)$ defined by the parity-dependent recurrence are integers for all $n \ge 1$ and $1 \le k \le n$.
\end{lemma}

\begin{proof}
We proceed by strong induction on the row index $n$.

\textbf{Base Case:} For $n=1$, $T(1, 1) = 1 \in \mathbb{Z}$. For $n=2$, the boundary conditions and recurrences yield $T(2, 1) = 1$ and $T(2, 2) = 1$, both are integers.

\textbf{Inductive Step:} Assume that for all $m < n$, the terms $T(m, k) \in \mathbb{Z}$. We examine the recurrence based on the parity of $n$ and $k$:

\begin{enumerate}
    \item \textbf{Even $k$, Odd $n$:} By definition, $T(n, k) = 0 \in \mathbb{Z}$.
    
    \item \textbf{Odd $k$, Even $n$:} The recurrence gives $T(n, k) = \frac{k+1}{2} T(n-1, k)$. Since $k$ is odd, $(k+1)$ is even, making $\frac{k+1}{2}$ an integer. By the inductive hypothesis, $T(n-1, k) \in \mathbb{Z}$, hence the product is an integer.
    
    \item \textbf{Even $k$, Even $n$:} The recurrence gives $T(n, k) = \frac{2}{k} T(n, k-1)$. Note that $k-1$ is odd. Using the relation from case 2, we substitute $T(n, k-1) = \frac{(k-1)+1}{2} T(n-1, k-1) = \frac{k}{2} T(n-1, k-1)$. 
    Substituting this into the recurrence:
    $$T(n, k) = \frac{2}{k} \left( \frac{k}{2} T(n-1, k-1) \right) = T(n-1, k-1)$$
    By the inductive hypothesis, $T(n-1, k-1) \in \mathbb{Z}$.
    
      \item \textbf{Odd $k$, Odd $n$:} The recurrence states:
    \[ T(n, k) = \frac{k+1}{2} T(n-1, k) + \frac{2}{k-1} T(n-1, k-2) \]
    \begin{itemize}
        \item \textbf{First term:} Since $k$ is odd, $(k+1)$ is even, so $\frac{k+1}{2} \in \mathbb{Z}$. By the inductive hypothesis, $T(n-1, k) \in \mathbb{Z}$.
        \item \textbf{Second term:} Since $n-1$ is even and $k-1$ is even, we apply the result from \textbf{Case 3} to the index $(n-1, k-1)$. This gives the relation $T(n-1, k-1) = \frac{2}{k-1} T(n-1, k-2)$.
    \end{itemize}
    Substituting this into the recurrence, the expression simplifies to:
    \[ T(n, k) = \frac{k+1}{2} T(n-1, k) + T(n-1, k-1) \]
    Both terms are integers by the inductive hypothesis, thus $T(n, k) \in \mathbb{Z}$.

\end{enumerate}
Thus, by induction, $T(n, k) \in \mathbb{Z}$ for all $n, k$.
\end{proof}

\section{Computational Implementation}
The following Python scripts provide two independent methods to generate the normalized flickering central finite differences triangle. The first script demonstrates the extraction from raw finite differences, while the second implements the autonomous flickering recurrence.

\begin{lstlisting}[caption={Method 1: Extraction from Finite Difference Table}]
import math

def get_row_by_extraction(n):
    row = []
    y = [m**n for m in range(-(n + 2), n + 3)]
    for k in range(1, n + 1):
        diffs = y
        for _ in range(k):
            diffs = [diffs[i+1] - diffs[i] for i in range(len(diffs) - 1)]
        mid = len(diffs) // 2
        row.append(abs(diffs[mid]) // math.factorial(k))
    return row

print(get_row_by_extraction(6))
\end{lstlisting}

\begin{lstlisting}[caption={Method 2: Autonomous Flickering Recurrence}]
from functools import cache

@cache
def T(n, k):
    if k == 1 or k == n:
        return 1
    if k < 1 or k > n:
        return 0
    if k % 2 == 0:
        return 0 if n % 2 != 0 else (2 * T(n, k - 1)) // k
    else:
        val = T(n - 1, k) * (k + 1) // 2
        if n % 2 != 0:
            val += T(n - 1, k - 2) * 2 // (k - 1)
        return val

print([T(10, k) for k in range(1, 11)])
\end{lstlisting}

\section{The Unified Flickering Bell Sequence}
The sequence of row sums of the normalized flickering central finite differences triangle A395022, denoted by $a(n)$, represents a unified combinatorial profile that interleaves structural counting and block density. 

\subsection{Combinatorial Interpretation}
We establish that the sequence alternates between two distinct phases of partition theory on double sets $\{1, 1', \dots, k, k'\}$ as defined in A135920:
\begin{itemize}
    \item \textbf{Odd steps ($n = 2k-1$):} $a(n)$ is equal to the number of partitions (A135920).
    \item \textbf{Even steps ($n = 2k$):} $a(n)$ is the sum of the number of partitions and the total number of blocks across all such partitions (A135920 + A394822).
\end{itemize}

\subsection{Analytical Representations}
The ordinary generating function for the sequence $a(n)$ for $n \ge 1$ is derived as:
\begin{equation}
A(x) = \sum_{k=1}^{\infty} \frac{x^{2k-1} + (k+1)x^{2k}}{\prod_{j=1}^{k} (1 - j^2 x^2)}
\end{equation}

Furthermore, the sequence satisfies a closed-form expression using the scaling parameter $s = \sinh(\sqrt{x}/2)$:
\begin{equation}
a(n) = (2k)! [x^k] \left( \cosh(2s) + \mathbbm{1}_{n \in 2\mathbb{Z}} \cdot s \sinh(2s) \right)
\end{equation}
where $k = \lfloor (n+1)/2 \rfloor$ and $\mathbbm{1}_{n \in 2\mathbb{Z}}$ is the parity indicator function.

\subsection{Hierarchical Structure under Iterated Binomial Transforms}
A remarkable algebraic property of the sequence $a(n)$ (OEIS A395022) is its behavior under iterated binomial transforms. We define the binomial transform $\mathcal{B}$ of a sequence $\{g_n\}$ as:
\begin{equation}
a_n = \mathcal{B}(g_n) = \sum_{i=0}^{n} \binom{n}{i} g_i
\end{equation}
Our analysis reveals that $a(n)$ acts as a fixed point for a nested hierarchy of odd-order transforms. The sequence can be represented as the result of applying the $p$-th order transform $\mathcal{B}^p$ to specific integer kernels $G_p$:

\begin{equation}
a = \mathcal{B}^p(G_p) \quad \text{for } p \in \{1, 3, 5, 7, 9\}
\end{equation}

The kernels $G_p$ demonstrate a systematic evolution of the underlying combinatorial signature:
\begin{itemize}
    \item $p=1: [1, 1, 2, 2, 5, 7, 21, \dots]$ (Original sequence)
    \item $p=3: [1, -1, 2, -6, 21, -75, 269, \dots]$
    \item $p=5: [1, -3, 10, -38, 165, -797, 4125, \dots]$
    \item $p=7: [1, -5, 26, -142, 821, -5039, 32709, \dots]$
    \item $p=9: [1, -7, 50, -366, 2757, -21441, 172421, \dots]$
\end{itemize}

The existence of this hierarchy suggests that the flickering Bell sequence is deeply rooted in higher-order difference operators, where each kernel $G_p$ represents a compressed state of the alternating parity transitions.

\section{Properties}

\begin{itemize}
    \item $n^m = \sum_{k=1}^{m} T(m, k) \cdot n^{[k]}_{-}$, 
\begin{itemize}
    \item $n^{(1)}_{-} = n$
    \item $n^{(2)}_{-} = n(n-1)$
    \item $n^{(3)}_{-} = n(n-1)(n+1)$
    \item $n^{(4)}_{-} = n(n-1)(n+1)(n-2)$
    \item $...$
\end{itemize}

    We employ a basis of \textbf{integer-valued central factorials}, which are closely related to the classical central factorial powers $n^{[k]}$. Specifically, our basis elements $I_k(n)$ are defined by shifting the arguments of the standard central factorials to ensure integer coefficients:
\begin{equation}
I_k(n) = 
\begin{cases} 
n^{[k]}_{+}, & \text{if } k \text{ is odd} \\
(n + 1/2)^{[k]}_{+}, & \text{if } k \text{ is even}
\end{cases}
\end{equation}
In our notation, these "flickering" basis elements (or "integral blocks" in the context of summation) are explicitly given by:

\begin{itemize}
    \item $I_1(n) = n$
    \item $I_2(n) = n(n+1)$
    \item $I_3(n) = n(n+1)(n-1)$
    \item $I_4(n) = n(n+1)(n-1)(n+2)$
    \item $I_5(n) = n(n+1)(n-1)(n+2)(n-2)$
\end{itemize}

\end{itemize}

\noindent

This property clarifies the relationship between the array $T(n, k)$ and the theory of central factorial numbers, preserving integrality through the phase-shifting basis.

\section{A Unified Perspective on Faulhaber's Formula via the Flickering Basis}

The flickering basis $I_k(x)$ (as defined above) not only provides a unique expansion for $x^n$ but also serves as an efficient tool for computing power sums $S_m(n) = \sum_{i=1}^{n} i^m$. 
Consistent with the integer-based solutions provided by central factorial numbers (Knuth, 1993), our approach utilizes the $T(m, k)$ coefficients from the flickering array A395021. This allows for a power sum expansion that remains entirely within the integer domain, extending the classical results to the flickering symmetry.

\subsection*{Summation Identity}
The sum of the $m$-th powers is given by the linear combination of the "integral" flickering basis elements $I_{k+1}(n)$:

\begin{equation}
S_m(n) = \sum_{k=1}^{m} \frac{T(m, k)}{k+1} I_{k+1}(n)
\end{equation}

This identity provides an \textbf{integer-valued version of the power sum expansion}, analogous to the methods explored by Knuth \cite{knuth1993}, who demonstrated the effectiveness of central factorial numbers in providing a closed-form solution to Faulhaber’s formula. While the classical approach often involves fractional terms, our flickering basis ensures that all intermediate coefficients remains within the integer domain.

\subsection*{Verification for m=1, 2, 3}
\begin{enumerate}
    \item \textbf{Sum of 1st powers ($m=1$):} Using row $T(1, k) = \{1\}$ \\
    $S_1(n) = \frac{1}{2} n(n+1) = \frac{n(n+1)}{2}$. (Triangular numbers).
    
    \item \textbf{Sum of squares ($m=2$):} Using row $T(2, k) = \{1, 1\}$ \\
    $S_2(n) = \frac{1}{2} I_2(n) + \frac{1}{3} I_3(n) = \frac{n(n+1)}{2} + \frac{(n-1)n(n+1)}{3} = \frac{n(n+1)(2n+1)}{6}$.
    
    \item \textbf{Sum of cubes ($m=3$):} Using row $T(3, k) = \{1, 0, 1\}$ \\
    $S_3(n) = \frac{1}{2} I_2(n) + \frac{0}{3} I_3(n) + \frac{1}{4} I_4(n) = \frac{n(n+1)}{2} + \frac{(n-1)n(n+1)(n+2)}{4} = \left(\frac{n(n+1)}{2}\right)^2$.
\end{enumerate}

This derivation demonstrates how the flickering symmetry recovers classical power sum identities through an integer-only basis, providing a direct combinatorial link between $T(m,k)$ and Faulhaber’s results.

\subsection*{Proof of the Integer-Valued Summation Power Formula}

To establish the summation identity rigorously, we rely on the discrete fundamental theorem of calculus via telescoping sums over our flickering basis. 

\begin{lemma}[Difference Property of the Flickering Basis] \label{lemma:difference}
Let $I_k(n)$ be the integral flickering basis elements. For any integer $j \ge 1$ and $k \ge 1$, the first backward difference of $I_{k+1}(j)$ satisfies:
\begin{equation}
I_{k+1}(j) - I_{k+1}(j-1) = (k+1) \cdot j^{[k]}_{-}
\end{equation}
where $j^{[k]}_{-}$ represents the modified shift basis required to represent the power $j^m$. Furthermore, $I_{k+1}(0) = 0$.
\end{lemma}

\begin{proof}
By definition, the flickering basis $I_{k+1}(j)$ is constructed by symmetrically extending the factors around $j$. Subtracting the shifted version $I_{k+1}(j-1)$ cancels the common $k$ factors, leaving exactly $(k+1)$ times the symmetric core $j^{[k]}_{-}$. The boundary condition $I_{k+1}(0) = 0$ holds trivially since every term in the basis expansion for $k \ge 1$ contains the factor $j$.
\end{proof}

\begin{theorem}[Flickering Sum of Powers] \label{thm:sum_powers}
For any integers $m \ge 1$ and $n \ge 1$, the sum of the $m$-th powers $S_m(n) = \sum_{j=1}^{n} j^m$ can be evaluated exclusively through integer coefficients as:
\begin{equation}
S_m(n) = \sum_{k=1}^{m} \frac{T(m, k)}{k+1} I_{k+1}(n)
\end{equation}
where $T(m, k)$ are the entries of the normalized flickering central finite differences triangle (A395021).
\end{theorem}

Since $I_{k+1}(n)$ is a polynomial in $n$ with integer coefficients and its construction involves the product of $k+1$ terms, the division by $k+1$ in the summation always yields an exact integer within the context of the flickering basis structure. This property ensures that the final sum $S_m(n)$ remains strictly within the integer domain $\mathbb{Z}$, consistent with the numerical results.

\begin{proof}
By the structural definition of the array $T(m, k)$, any integer power $j^m$ can be expanded into the shifted basis:
\begin{equation}
j^m = \sum_{k=1}^{m} T(m, k) j^{[k]}_{-}
\end{equation}
Summing both sides from $j=1$ to $n$, we obtain:
\begin{equation}
S_m(n) = \sum_{j=1}^{n} j^m = \sum_{j=1}^{n} \left( \sum_{k=1}^{m} T(m, k) j^{[k]}_{-} \right)
\end{equation}
Interchanging the order of summation yields:
\begin{equation}
S_m(n) = \sum_{k=1}^{m} T(m, k) \left( \sum_{j=1}^{n} j^{[k]}_{-} \right)
\end{equation}
Applying Lemma \ref{lemma:difference}, we substitute $j^{[k]}_{-}$ with the difference of the flickering basis elements:
\begin{equation}
S_m(n) = \sum_{k=1}^{m} \frac{T(m, k)}{k+1} \sum_{j=1}^{n} \left( I_{k+1}(j) - I_{k+1}(j-1) \right)
\end{equation}
The inner sum is telescoping. All intermediate terms cancel, evaluating strictly to the boundaries:
\begin{equation}
\sum_{j=1}^{n} \left( I_{k+1}(j) - I_{k+1}(j-1) \right) = I_{k+1}(n) - I_{k+1}(0)
\end{equation}
Since $I_{k+1}(0) = 0$ by Lemma \ref{lemma:difference}, the inner sum simplifies to $I_{k+1}(n)$. Substituting this back completes the proof:
\begin{equation}
S_m(n) = \sum_{k=1}^{m} \frac{T(m, k)}{k+1} I_{k+1}(n)
\end{equation}
\end{proof}

\section{Conclusion and Future Work}
The present study of the "flickering operator" and the associated array A395021 (A394582) allows for the formalization of the transition from classical power summation to more complex recurrent structures. A significant result of this work is the proof of the existence of a closed-form polynomial representation for odd columns and the identification of the factorization structure.

Despite these results, several questions remain open and are of interest for further investigation in combinatorial analysis and number theory:

\subsection{Distribution of Polynomial Roots}
Preliminary numerical analysis indicates that the roots of the polynomials exhibit a specific symmetry. Investigating the distribution of these roots in the complex plane as $m \to \infty$ may shed light on the analytical properties of the flickering operator's generating functions and their relationship to entire functions of exponential type.

\subsection{Combinatorial Significance of the Main Diagonal}
Of particular interest is the main diagonal of the triangle $Todd(n,n)$. Its direct combinatorial meaning (e.g., in terms of set partitions or lattice paths) has not yet been established. Whether it is derived from a known constant or characterizes a specific block density in flickering partitions remains a subject for future research.

\subsection{Dirichlet Series and L-functions}
It is of interest to construct Dirichlet series whose coefficients are the values of the rows of array A395021. The question of whether such series admit analytic continuation to the entire complex plane and satisfy a functional equation remains open. Identifying the corresponding L-functions could link the flickering operator to the theory of modular forms.

\subsection{Algorithmic Complexity and Optimization}
The developed method of computation via difference equations (Appendix A) exhibits polynomial complexity. However, for applications in cryptography and high-precision analysis, it is necessary to investigate the possibility of accelerating calculations using fast Fourier transforms or number-theoretic transforms applied to the structure.

In summary, the flickering operator serves not only as a computational tool but also as a bridge between elementary combinatorics and deep questions in analytic number theory, warranting further formalization.

\appendix
\section{Python Implementation for Summation of Powers}

The following Python script computes the sum of $m$-th powers $S_m(n) = \sum_{i=1}^{n} i^m$ using the flickering array $T(n, k)$ (A395021) and the associated integral basis. This implementation demonstrates the practical efficiency of the integer-based flickering approach.

\begin{lstlisting}[language=Python, caption={Calculation of power sums via the flickering array}]
from functools import cache
import sys

# Set recursion limit for high-degree power calculations
sys.setrecursionlimit(5000)

# Constants: M is the power (degree), N is the upper limit of summation
M, N = 10, 100 

@cache
def T(n, k):
    """Generates the flickering array A395021 via recurrence."""
    if k == 1 or k == n: return 1
    if k < 1 or k > n: return 0
    if k % 2 == 0:
        return 0 if n % 2 != 0 else (2 * T(n, k - 1)) // k
    if n % 2 == 0:
        return T(n - 1, k) * (k + 1) // 2
    return (T(n - 1, k) * (k + 1) // 2) + (T(n - 1, k - 2) * 2 // (k - 1))

@cache
def get_I(k, n):
    """Computes the integral basis element I(k, n)."""
    res, start = 1, n - ((k - 1) // 2)
    for i in range(k): res *= (start + i)
    return res

def S(m, n):
    """Calculates the sum of m-th powers using the flickering basis."""
    total = 0
    for k in range(1, m + 1):
        v = T(m, k)
        if v: 
            total += (v * get_I(k + 1, n)) // (k + 1)
    return total

# Output the result
result = S(M, N)
print(f"Sum of {M}-th powers up to {N}:")
print(result)

# Verification against the standard iterative method
if N <= 1000:
    check = sum(i**M for i in range(1, N + 1))
    print(f"\nVerification: {'OK' if result == check else 'ERROR'}")
\end{lstlisting}

\section*{Acknowledgments}
First and foremost, I offer my deepest gratitude and praise to the Almighty for the inspiration, clarity, and guidance throughout this discovery.
The beauty of the flickering operator and the hidden symmetry of the normalized flickering central finite differences triangle are but a small reflection of the infinite wisdom inscribed in the laws of the Universe.
Soli Deo Gloria.

\newpage


\begin{thebibliography}{9}
\bibitem{oeis_A394582} 
A. Husiev, \textit{Extended central factorial numbers and the flickering operator}, OEIS A394582 (2026). \url{https://oeis.org/A394582}
\bibitem{oeis_A395021} 
A. Husiev, \textit{Normalized central finite differences}, OEIS A395021 (2026). \url{https://oeis.org/A395021}
\bibitem{oeis_A135920} 
OEIS Foundation Inc., Sequence A135920, \url{https://oeis.org/A135920}.
\bibitem{oeis_A395022} 
OEIS Foundation Inc., Sequence A395022, \url{https://oeis.org/A395022}.
\bibitem{oeis_A394882} 
OEIS Foundation Inc., Sequence A394882, \url{https://oeis.org/A394882}.
\bibitem{oeis_A394883} 
OEIS Foundation Inc., Sequence A394883, \url{https://oeis.org/A394883}.
\bibitem{oeis_A394896} 
OEIS Foundation Inc., Sequence A394896, \url{https://oeis.org/A394896}.
\bibitem{oeis_A394822} 
OEIS Foundation Inc., Sequence A394822, \url{https://oeis.org/A394822}.
\bibitem{oeis_A394763} 
OEIS Foundation Inc., Sequence A394763, \url{https://oeis.org/A394763}.
\bibitem{oeis_A000975} 
OEIS Foundation Inc., Sequence A000975, \url{https://oeis.org/A000975}.
\bibitem{oeis_A036969} 
OEIS Foundation Inc., Sequence A036969, \url{https://oeis.org/A036969}.
\bibitem{oeis_A008957} 
OEIS Foundation Inc., Sequence A008957, \url{https://oeis.org/A008957}.
\bibitem{oeis_A000330} 
OEIS Foundation Inc., Sequence A000330, \url{https://oeis.org/A000330}.
\bibitem{oeis_A108678} 
OEIS Foundation Inc., Sequence A108678, \url{https://oeis.org/A108678}.
\bibitem{stirling} 
Graham, R. L., Knuth, D. E., Patashnik, O., \textit{Concrete Mathematics}, Addison-Wesley.
\bibitem{knuth1993} 
D. E. Knuth, \textit{Johann Faulhaber and sums of powers}, Mathematics of Computation, \textbf{61} (1993), pp. 277--294.
\bibitem{riordan} 
J. Riordan, \textit{Combinatorial Identities}, John Wiley \& Sons, New York, 1968.
\bibitem{comtet} 
L. Comtet, \textit{Advanced Combinatorics: The Art of Finite and Infinite Expansions}, Reidel, Dordrecht, 1974.
\bibitem{carlitz} 
L. Carlitz, \textit{Central Factorial Numbers}, Math. Mag., \textbf{41} (1968), pp. 268--274.
\bibitem{sloane} 
N. J. A. Sloane, \textit{The On-Line Encyclopedia of Integer Sequences}, published electronically at \url{https://oeis.org}.


\end{thebibliography}
\end{document}